\documentclass[11pt,psamsfonts]{amsart}
\usepackage{amsmath}
\usepackage{amsthm}
\usepackage{amssymb}
\usepackage{amscd}
\usepackage{amsfonts}
\usepackage{amsbsy}
\usepackage{graphicx}
\usepackage[dvips]{psfrag}
\usepackage{array}
\usepackage{color}
\usepackage{epsfig}
\usepackage{url}
\usepackage{ulem}

\newcolumntype{L}{>{\displaystyle}l}
\newcolumntype{C}{>{\displaystyle}c}
\newcolumntype{R}{>{\displaystyle}r}

\newcommand{\R}{\ensuremath{\mathbb{R}}}

\newcommand{\Z}{\ensuremath{\mathbb{Z}}}

\newcommand{\CF}{\ensuremath{\mathcal{F}}}

\newcommand{\CO}{\ensuremath{\mathcal{O}}}
\newcommand{\ov}{\overline}

\newcommand{\T}{\theta}
\newcommand{\f}{\varphi}
\newcommand{\al}{\alpha}

\newcommand{\s}{\ensuremath{\mathbb{S}}}
\newcommand{\A}{\ensuremath{\mathcal{A}}}

\newcommand{\X}{\ensuremath{\mathcal{X}}}

\def\p{\partial}
\def\e{\varepsilon}

\newtheorem {theorem} {Theorem}
\newtheorem {proposition} [theorem] {Proposition}

\newtheorem {lemma} [theorem] {Lemma}

\newtheorem {claim} {Claim}

\begin{document}

\title[Averaging theory at any order for piecewise differential systems]
{Averaging theory at any order for computing limit cycles of discontinuous piecewise differential systems with many zones}

\author[J. Llibre, D.D. Novaes and C.A.B. Rodrigues]
{Jaume Llibre, Douglas D. Novaes  and Camila A. B. Rodrigues}

%
%
%


%

\let\thefootnote\relax\footnotetext{\\ Jaume Llibre\\
Departament de Matematiques,
Universitat Aut\`{o}noma de Barcelona, 08193 Bellaterra,\\
Barcelona, Catalonia, Spain\\
Email: jllibre@mat.uab.cat}

\let\thefootnote\relax\footnotetext{\\ Douglas D. Novaes\\
Departamento de Matem\'{a}tica, Universidade Estadual de Campinas,\\
Rua S\'{e}rgio Baruque de Holanda, 651, Cidade Universit\'{a}ria Zeferino Vaz, 13083-859,\\
Campinas, S\~{a}o Paulo, Brazil,\\
Email: ddnovaes@ime.unicamp.br}

\let\thefootnote\relax\footnotetext{\\ Camila A. B. Rodrigues\\
Departamento de Matem\'{a}tica, ICMC-Universidade de S\~{a}o Paulo, 13560-970,\\
S\~{a}o Carlos, S\~{a}o Paulo, Brazil,\\
Email: camilaap@icmc.usp.br}
\maketitle

\noindent{\bf Abstract}
This work is devoted to study the existence of periodic solutions for a family of planar discontinuous differential systems $Z(x,y;\e)$ with many zones. We show that for $|\e|\neq0$ sufficiently small the averaged functions at any order control the existence of crossing limit cycles for systems in this family. We also provide some examples dealing with nonlinear centers when $\e=0$.

\smallskip


\smallskip

\noindent {\bf Mathematics Subject Classi cation (2000)} 34C29  $\cdot$  34C25  $\cdot$  37G15 $\cdot$  34C07

\section{Introduction and statement of the main results}

In the qualitative theory of real planar differential system the determination
of limit cycles, defined by Poincar\'{e} \cite{poincare}, has become one
of the main problems. The second part of the $16$th Hilbert problem deals with planar polynomial vector fields and proposes to find a uniform upper  bound $H(n)$ (called Hilbert's number) for the number of limit cycles
that these vector fields can have depending only on the polynomial degree $n$.
The averaging method has been used to provide lower bounds for the Hilbert number $H(n)$ see, for instance, \cite{LMT}. The interest on this topic extends to what we call  {\it discontinuous piecewise vector fields}.

\smallskip

The increasing interest in the theory of nonsmooth vector fields has been mainly motivated by its strong relation with Physics, Engineering, Biology, Economy, and other branches of science. In fact, their associated differential systems are very useful to model phenomena presenting abrupt switches such as electronic relays, mechanical impact, and neuronal networks, see for instance \cite{BBCK, Co, physDspecial}. The extension of the averaging theory to discontinuous piecewise vector field has been the central subject of investigation of the following works  \cite{LliItiNov2015,LliMerNovJDE2015,LliNovPreprint2014,LliNovTeiBSM2015}.

\smallskip

A piecewise vector field defined on an open bounded set $U\subset\R^n$ is a function $F:U\rightarrow\R^n$ which is continuous except on a set $\Sigma$ of measure $0$, called the {\it set of discontinuity} of the vector field $F$. It is assumed that $U\setminus\Sigma$ is a finite collection of disjoint open sets $U_i,$ $i=1,2,\ldots,m,$ such that the restriction $F_i=F\big |_{U_i}$ is continuous and extendable to the compact set $\ov{U_i}$. The local trajectory of $F$ at a point $p\in U_i$  is given by the usual notion. However the local trajectory of $F$ at a point $p\in\Sigma$ needs to be given with some care. In \cite{Filippov}, taking advantage of the theory of differential inclusion (see \cite{AC}), Filippov established some conventions for what would be a local trajectory at points of discontinuity where the set $\Sigma$ is locally a codimension one embedded submanifold of $\R^n.$ For a such point $p \in \Sigma$, we consider a sufficiently small neighborhood $U_p$ of $p$ such that $\Sigma$ splits $U_p\setminus\Sigma$ in two disjoint open sets $U_p^+$ and $U_p^-$ and denote $F^{\pm}(p)=F\big|_{U_p^{\pm}}(p).$ In short, if the vectors $F^{\pm}(p)$ point at the same direction then the local trajectory of $F$ at $p$ is given as the concatenation of the local trajectories of $F^{\pm}$ at $p$. In this case we say that the trajectory {\it crosses} the set of discontinuity and that $p$ is a {\it crossing point}. If the vectors $F^{\pm}(p)$ point in opposite directions then the local trajectory of $F$ at $p$ slides on $\Sigma$. In this case we say that $p$ is a {\it sliding point}. For more details on the Filippov conventions see \cite{Filippov,GST}.

\smallskip

In this paper we are interested in establishing conditions for the existence of crossing limit cycles for a class of planar discontinuous piecewise vector fields, that is limit cycles which only crosses the set of discontinuity $\Sigma$. It is worth to say that if $\Sigma$ is locally described as $h^{-1}(0)$, being $h:U\rightarrow \R$ a smooth function and $0$ a regular value, then $\langle\nabla h(p), F^+(p)\rangle\langle\nabla h(p), F^-(p)\rangle>0$ is the condition in order that $p$ is a crossing point.

\smallskip

In the sequel we introduce a short review of the averaging theory for computing isolated periodic solutions of differential equations. Then we set the class of planar discontinuous piecewise differential equations that we are interested. After that the rest of the section is devoted to the statement of our main result.

\subsection{Background on the averaging theory for smooth systems}

Let $D$ be an open bounded subset of $\R_+$ and consider  $C^{k+1}$ functions  $F_i:\R\times D\rightarrow\R$ for  $i=1,2,\ldots,k$, and $R:\R\times D\times(-\e_0,\e_0)\rightarrow\R$. We assume that all these functions are $2\pi$-periodic in the first variable. Now consider the following differential equation
\begin{equation}\label{smooth}
r'(\T)=\sum_{i=0}^k\e^i F_i(\T,r)+\e^{k+1} R(\T,r,\e),
\end{equation}
and assume that the solution $\f(\T,z)$ of the {\it unperturbed system} $r'(\T)= F_0(\T,r),$ such that $\f(0,\rho)=\rho,$ is $2\pi$-periodic for every $\rho\in D$. Here the prime denotes the derivative in the variable $\T$.

\smallskip

A central question in the study of system \eqref{smooth} is to understand which periodic orbits of the unperturbed system $r'(\T)= F_0(\T,r)$ persists for $|\e|\neq0$ sufficiently small. In others words to provide sufficient conditions for the persistence of isolated periodic solutions. The averaging theory is one of the best tools to track this problem.  Summarizing, it consists in defining a collection of functions $f_i:D\rightarrow\R,$ for $i=1,2,\ldots,k$,  called {\it averaged functions}, such that their simple zeros provide the existence of isolated periodic solutions of the differential equation \eqref{smooth}.  In \cite{LliNovTeiN2014,LliNovTeiN2014c} it was proved that these averaged functions are
\begin{equation}\label{f}
f_i(\rho)=\dfrac{y_i(2\pi,\rho)}{i!},
\end{equation}
where $y_i:\R\times D\rightarrow \R$ for $i=1,2,\ldots,k$, are
defined recurrently by the following integral equations
\begin{equation}
\label{y}
\begin{array}{RL}
y_1(\T,\rho)=&\int_0^{\T} F_1\left(s,\f(s,\rho)\right)ds,\vspace{0.3cm}\\

y_i(\T,\rho)=&i!\int_0^{\T}\Big(F_i\left(s,\f(s,\rho)\right)
+\sum_{l=1}^{i}\sum_{S_l} \dfrac{1}{b_1!\,b_2!2!^{b_2}\cdots
b_l!l!^{b_l}}\vspace{0.2cm}\\
&\cdot\p^L F_{i-l} \left(s,\f(s,\rho)\right)
\prod_{j=1}^ly_j(s,\rho)^{b_j}\Big)ds, \text{ for }\, i=2,\ldots,k.
\end{array}
\end{equation}
Here $\p^L G(\phi,\rho)$ denotes the derivative order $L$ of a function $G$ with respect to the variable $\rho$,  and $S_l$ is the set of all $l$-tuples of non-negative integers
$(b_1,b_2,\ldots,b_l)$ satisfying $b_1+2b_2+\cdots+lb_l=l$, and
$L=b_1+b_2+\cdots+b_l$.

\subsection{A class of planar discontinuous piecewise smooth vector fields}

When one consider the above problem in the world of discontinuous piecewise differential systems  it is not always true that the higher averaged functions \eqref{f} allow to study the persistence of isolated periodic solutions. In \cite{LliNovTeiBSM2015, LliMerNovJDE2015} this problem was considered for general Filippov systems when $F_0(\T,r) \equiv 0$ and it was proved that the averaged function of first order can provide information in the persistence of crossing isolated periodic solutions. Furthermore the authors have found conditions on those systems in order to assure that the averaged function of second order also provides information on the existence of crossing isolated periodic solutions.  When $F_0(\T,r) \not\equiv 0$ but satisfies the condition that the solution $\f(\T,\rho)$ is $2\pi$-periodic the authors in \cite{LliNovPreprint2014} have found conditions on those systems in order to assure that the averaged function of first order provides information on the existence of crossing isolated periodic solutions.

\smallskip

This work is devoted to study the existence of isolated periodic solutions for an $\e$-family of planar discontinuous piecewise differential system $(\dot x, \dot y)^T=Z(x,y;\e)$. Here the dot denotes derivative in the variable $t.$ In short we shall provide sufficient conditions in order to show that for $|\e|\neq0$ sufficiently small the averaged functions \eqref{f} at any order can be used for obtaining information on the existence of crossing limit cycles for systems of this family.

\smallskip

We start defining the family of smooth piecewise differential systems that we shall study. The construction that we shall perform in the sequel has been done in \cite{LliMerNovJDE2015} for a particular class of systems. Let $n>1$ be a positive integer, $\al_n=2\pi$ and $\al=(\al_0,\al_1,\ldots,\al_{n-1})\in\mathbb{T}^n$ is a $n$-tuple of angles such that $0=\al_0<\al_1<\al_2<\cdots<\al_{n-1}< \al_n=2\pi$
and let $\X(x,y;\e)=(X_1,X_2,\ldots,X_n)$ be a $n$-tuple of smooth vector fields defined on an open bounded neighborhood $U\subset\R^2$ of the origin and depending on a small parameter $\e$ in the following way
\begin{equation}\label{Xj}
X_j(x,y;\e)=\sum_{i=0}^k \e^i X_i^j(x,y) \quad\text{for}\quad j=1,2,\ldots,n.
\end{equation}
For $j=1,2,\ldots,n$ let $L_j$ be the intersection between the domain $U$ with the
ray starting at the origin and passing through the point $(\cos
\al_j,\sin \al_j)$, and take  $\Sigma=\bigcup_{j=1}^n L_j.$ We note that $\Sigma$
splits the set $U\backslash \Sigma\subset\R^2$ in $n$ disjoint open
sectors. We denote the sector delimited by $L_j$ and $L_{j+1}$, in counterclockwise sense, by
$C_j$, for $j=1,2,\ldots,n.$

\smallskip

Now let $Z_{\X,\al}:U\rightarrow\R^2$ be a discontinuous piecewise vector field defined as $Z_{\X,\al}(x,y;\e)$ $=X_j(x,y;\e)$ when $(x,y)\in C_j$, and consider the following planar discontinuous piecewise differential system
\begin{equation}\label{planar}
(\dot x,\dot y)^T=Z_{\X,\al}(x,y;\e).
\end{equation}
The above notation means that at each sector $C_j$ we are considering the smooth differential system
\begin{equation}\label{planarsector}
(\dot x,\dot y)^T=\X_j(x,y;\e).
\end{equation}

As our main hypothesis we shall assume that there exists a period annulus $\A$ homeomorphic to $\{(x,y)\in U: 0<|(x,y)|\leq 1\}$, surrounding the origin, fulfilled by crossing periodic solutions of the {\it unperturbed system} $(\dot x,\dot y)^T=Z_{\X,\al}(x,y;0)$.

\subsection{Standard form and main result}

The averaging theory deals with periodic nonautonomous differential systems in the standard form \eqref{smooth}.  Therefore in order to use the averaging theory for studying system \eqref{planar} it has to be written in the standard form. A possible approach for doing this is to consider the polar change of variables $x=r\,\cos\T$ and  $y=r\,\sin\T$. However the appropriate change of variables may depend on the initial system \eqref{planar}. In general, for each $j=1,2,\ldots,n$, after a suitable change of variables system \eqref{planarsector} reads
\begin{equation}
\label{star}
r'(\T)=\dfrac{\dot r(t)}{\dot \T(t)}=\sum_{i=0}^k\e^i F_i^j(\T,r) + \e^{k+1}R^j(\T,r,\e).
\end{equation}
Now $\T \in [\al_{j-1},\al_j],$  $F_i^j:\mathbb{S}^1\times D\rightarrow\R$ and $R^j:\R\times D\times(-\e_0,\e_0)\rightarrow\R$ are $C^{k+1}$ functions depending on the vector fields $X_i^j$, and they are $2\pi$-periodic in the first variable, being $D$ an open bounded interval of $\R_+$ and $\mathbb{S}^1\equiv\R/(2\pi\Z)$.  Furthermore system \eqref{planar} becomes
\begin{equation}\label{s1}
r'(\T)=\sum_{i=0}^k\e^i F_i(\T,r)+\e^{k+1} R(\T,r,\e),
\end{equation}
where
\begin{equation}\label{funcF}
\begin{array}{l}
\displaystyle F_i(\T,r)=\sum_{j=1}^{n}\chi_{[\al_{j-1},\al_j]}(\T) F_i^j(\T,r),\,\, i=0,1,...,k,\quad \text{and}\vspace{0.1cm}\\
\displaystyle R(\T,r,\e)=\sum_{j=1}^{n}\chi_{[\al_{j-1},\al_j]}(\T)R^j(\T,r,\e),
\end{array}
\end{equation}
where the characteristic function $\chi_A(\T)$ of an interval $A$ is defined as
\begin{equation*}
\chi_{A}(\T)=
\begin{cases}
1 & \text{if $\T \in A$},\\
0 & \text{if $\T \not\in A$}.
\end{cases}
\end{equation*}X
System \eqref{s1} is now a nonautonomous periodic discontinuous piecewise differential system having its set of discontinuity formed by $\Sigma=(\{\T=0\}\cup\{\T=\al_1\}\cup\cdots \cup\{\T=\al_{n-1}\})\cap \mathbb{S}^1 \times D$.

\smallskip

Denote by $\varphi(\T,\rho)$ the solution of the system $r'(\T)=F_0(\T,r)$  such that $\f(0,\rho)=\rho$. From now on this last system will be called {\it unperturbed system}. We assume the following hypothesis:
\begin{itemize}
\item[(H1)]  For each $z\in D$ the solution $\varphi(\T,\rho)$  is defined for every $\T\in\s^1$, it reaches $\Sigma$ only at crossing points, and it is $2\pi$-periodic.
\end{itemize}

In what follows we state our main result.

\begin{theorem}
\label{maintheorem}Assume that for some $l \in \{1,2,\ldots,k\}$ the functions defined in \eqref{f} satisfy $f_s=0$ for $s=1,2,\ldots,l-1$ and $f_l\neq0$. Moreover we assume that the function $\f(\T,\rho)$ of the unperturbed system is a $2\pi$-periodic function. If there exists $\rho^* \in D$ such that $f_l(\rho^*)=0$ and $f_l'(\rho^*)\neq0$, then for $|\e| \neq 0$ sufficiently small there exists a $2\pi$-periodic solution $r(\T,\e)$ of system \eqref{s1} such that $r(0,\e) \to \rho^*$ when $\e \to 0$.
\end{theorem}

The assumption $D\subset\R$ is not restrictive. In fact, if one consider $D$ as being an open subset of $\R^n$ the conclusion of Theorem \ref{maintheorem} still holds by assuming that the Jacobian matrix $J f_l(\rho^*)$ is nonsingular, that is $\det(J f_l(\rho^*))\neq0$. In this case the derivative $\p^L G(\phi,\rho)$ is a symmetric $L$-multilinear map which is applied to a ``product'' of
$L$ vectors of $\R^n$, denoted as $\prod_{j=1}^Ly_j\in
\R^{nL}$ (see \cite{LliNovTeiN2014}).

\smallskip

For the particular class of systems \eqref{s1} Theorem \ref{maintheorem} generalizes the main results of \cite{LliMerNovJDE2015,LliNovPreprint2014,LliNovTeiBSM2015}, increasing the order of the averaging theory. It also generalizes the main results of \cite{LliItiNov2015,WZ} dealing now with nonvanishing unperturbed systems and allowing more zones of continuity.

\smallskip

This paper is organized as follows. In section \ref{theaveragedfunctions} we provide, explicitly, the formulae of the averaged functions \eqref{f} for nonsmooth systems in the standard form \eqref{s1}. In section \ref{proofofthemainresult} we state some auxiliar results for proving Theorem \ref{maintheorem}.
In section \ref{examples} we use Theorem \ref{maintheorem} to give an estimative for the number of limit cycles of three types of planar systems: nonsmooth perturbation of a linear center, a nonsmooth perturbation of a discontinuous piecewise constant center, and a nonsmooth perturbation of an isocrhonous quadratic center.

\section{The averaged functions}
\label{theaveragedfunctions}
In this section we develop a recurrence to compute the averaged function \eqref{f} in the particular case of the discontinuous differential equation \eqref{s1}. So consider the functions $z_i^j:(\al_{j-1},\al_{j}]\times D \to \R$ defined recurrently for $i=1,2,\ldots,k$ and $j=1,2,\ldots,n,$ as

\begin{equation}\label{z11}
\begin{array}{RL}
z_1^1(\T,\rho)=&\int_0^{\T} \bigg( F_1^1(\phi,\f(\phi,\rho)) + \partial F_0^1(\phi,\f(\phi,\rho))z_1^1(\phi,\rho) \bigg)d\phi,\vspace{0.3cm}\\

z_i^1(\T,\rho)=& i!\int_{0}^{\T}\bigg(F_i^{1}(\phi,\f(\phi,\rho))\vspace{0.2cm}\\
&+\sum_{l=1}^{i}\sum_{S_l}\dfrac{1}{b_1!\,b_2!2!^{b_2}\cdots b_l!l!^{b_l}}\cdot\partial ^L F_{i-l}^{1}(\phi,\f(\phi,\rho)) \prod_{m=1}^l z_m^{1}(\phi,\rho)^{b_m} \bigg)d\phi,\vspace{0.3cm}\\

z_i^j(\T,\rho)=&z_i^{j-1}(\al_{j-1},\rho)+i! \int_{\al_{j-1}}^{\T}\bigg(F_i^{j}(\phi,\f(\phi,\rho))\vspace{0.2cm}\\
&+\sum_{l=1}^{i}\sum_{S_l}\dfrac{1}{b_1!\,b_2!2!^{b_2}\cdots b_l!l!^{b_l}}\cdot\partial ^L F_{i-l}^{j}(\phi,\f(\phi,\rho)) \prod_{m=1}^l z_m^{j}(\phi,\rho)^{b_m} \bigg)d\phi.
\end{array}
\end{equation}
Thus we have the next result.
\begin{proposition}
For $i=1,2,\ldots,k,$ the averaged function \eqref{f} of order $i,$  is
\begin{equation}
\label{fpromediada}
f_i(\rho)=\dfrac{z_i^n(2\pi,\rho)}{i!}.
\end{equation}
\end{proposition}
\begin{proof}For each $i=1,2,\cdots,k,$ define
\begin{equation}\label{zprom}
z_i(\T,\rho)=\sum_{j=1}^{n}\chi_{[\al_{j-1},\al_j]}(\T)z_i^j(\T,\rho).
\end{equation}
Given $\T \in [0,2\pi]$ there exists a positive integer $\bar{k}$ such that $\T \in (\al_{\bar{k}-1}, \al_{\bar{k}}]$ and, therefore $z_i(\T,\rho)=z_i^{\bar{k}}(\T,\rho)$. Moreover using the expressions \eqref{funcF} and \eqref{zprom} we can write \eqref{z11} into the form
\begin{equation}
\label{recurrencezi}
\begin{array}{RL}
z_1^1(\T,\rho)=&\int_0^{\T} \bigg( F_1(\phi,\f(\phi,\rho)) + \partial F_0(\phi,\f(\phi,\rho))z_1(\phi,\rho) \bigg)d\phi,\vspace{0.3cm}\\

z_i^{1}(\T,\rho)=&i! \int_{0}^{\T}\bigg(F_i(\phi,\f(\phi,\rho))\vspace{0.2cm}\\
&+\sum_{l=1}^{i}\sum_{S_l}\dfrac{1}{b_1!\,b_2!2!^{b_2}\cdots b_l!l!^{b_l}}\partial ^L F_{i-l}(\phi,\f(\phi,\rho)) \prod_{m=1}^l z_m(\phi,\rho)^{b_m} \bigg)d\phi,\vspace{0.3cm}\\

z_i^{\bar{k}}(\T,\rho)=&z_i^{\bar{k}-1}(\al_{\bar{k}-1},\rho)+i! \int_{\al_{\bar{k}-1}}^{\T}\bigg(F_i(\phi,\f(\phi,\rho))\vspace{0.2cm}\\
&+\sum_{l=1}^{i}\sum_{S_l}\dfrac{1}{b_1!\,b_2!2!^{b_2}\cdots b_l!l!^{b_l}}\partial ^L F_{i-l}(\phi,\f(\phi,\rho)) \prod_{m=1}^l z_m(\phi,\rho)^{b_m} \bigg)d\phi.
\end{array}
\end{equation}
In the above equality we are denoting
\[
\partial ^L F_{i-l}(\phi,\f(\phi,\rho)) =\sum_{j=1}^{n}\chi_{[\al_{j-1},\al_j]}(\phi)\partial ^L F_{i-l}^j(\phi,\f(\phi,\rho)).
\]

Proceeding recursively on $\bar{k}$ we obtain
\begin{equation}
\label{expressionz1}
\begin{array}{RL}
z_1(\T,\rho)=&\int_0^{\T} \bigg( F_1(\phi,\f(\phi,\rho)) + \partial F_0(\phi,\f(\phi,\rho))z_1(\phi,\rho) \bigg)d\phi,\vspace{0.3cm}\\
z_i(\T,\rho)=&\sum_{p=1}^{\bar{k}-1} \int_{\al_{p-1}}^{\al_p}\bigg(F_i^{p}(\phi,\f(\phi,\rho))+\sum_{l=1}^{i}\sum_{S_l}\dfrac{1}{b_1!\,b_2!2!^{b_2}\cdots b_l!l!^{b_l}}\vspace{0.2cm}\\
&\cdot\partial ^L F_{i-l}^{p}(\phi,\f(\phi,\rho)) \prod_{m=1}^l z_m^{p}(\phi,\rho)^{b_m} \bigg)d\phi +\int_{\al_{\bar{k}-1}}^{\T}\bigg(F_i^{\bar{k}}(\phi,\f(\phi,\rho))\vspace{0.2cm}\\
&+\sum_{l=1}^{i}\sum_{S_l}\dfrac{1}{b_1!\,b_2!2!^{b_2}\cdots b_l!l!^{b_l}}\partial ^L F_{i-l}^{\bar{k}}(\phi,\f(\phi,\rho)) \prod_{m=1}^l z_m^{\bar{k}}(\phi,\rho)^{b_m} \bigg)d\phi\vspace{0.3cm}\\
=&i!\int_{0}^{\T}\bigg(F_i(\phi,\f(\phi,\rho))+\sum_{l=1}^{i}\sum_{S_l}\dfrac{1}{b_1!\,b_2!2!^{b_2}\cdots b_l!l!^{b_l}}\vspace{0.2cm}\\
&\cdot\partial ^L F_{i-l}(\phi,\f(\phi,\rho)) \prod_{m=1}^l z_m(\phi,\rho)^{b_m} \bigg)d\phi.
\end{array}
\end{equation}

Computing the derivative in the variable $\T$ of the expressions \eqref{expressionz1} and \eqref{y} for $i=1$ we see that the functions $z_1(\T,\rho)$ and $y_1(\T,\rho)$ satisfy the same differential equation. Moreover for each $i=2,\cdots,k$, the integral equations \eqref{y} and \eqref{expressionz1} which provides respectively $y_i$ and $z_i$ are defined by the same recurrence. Therefore we conclude that $y_i$ and $z_i$ satisfy the same differential equations for $i=1,2,\cdots,k$, which are linear with variable coefficients (that is the Existence and Uniqueness Theorem holds). Now, it only remains  to prove that their initial conditions coincide. Let $i \in \{1,2,\ldots,k\}$, then $y_i(0,\rho)=0$ and by \eqref{recurrencezi} $z_i(0,\rho)=0$, concluding that the initial conditions are the same.  Hence $y_i(\T,\rho)=z_i(\T,\rho)$, conlcuding the Proposition. \end{proof}

\smallskip

Note that when $F_0\neq0$ the recurrence defined in \eqref{z11} is actually an integral equation. Moreover in order to implement an algorithm to compute the averaged function, it may be easier to write each $z_i^j$ in terms of the partial Bell polynomials, which are already implemented in algebraic manipulators as Mathematica and Maple. For each pair of nonnegative integers $(p,q)$, the partial {\it Bell polynomial} is defined as
\[
B_{p,q}(x_1,x_2,\ldots,x_{p-q+1})=\sum_{\widetilde{S}_{p,q}}\frac{p!}{b_1! b_2! \cdots b_{p-q+1}!} \prod_{j=1}^{p-q+1} \bigg(\frac{x_j}{j!}\bigg)^{b_j},
\]
where $\widetilde{S}_{p,q}$ is the set of all $(p-q+1)$-tuple of nonnegative integers $(b_1,b_2, \ldots, b_{p-q+1})$
satisfying $b_1 + 2b_2 + \cdots + (p-q + 1)b_{p-q+1} = p$, and $b_1 + b_2 + \cdots + b_{p-q+1} = q$. In the next proposition, following \cite{Novaes2016}, we solve the integral equation \eqref{z11} to provide the explicit recurrence formula for $z_i^j$ in terms of the Bell polynomials.
\begin{proposition}
\label{lemmabell}
For each $j=1,2,\ldots,n$ let $\eta_{j}(\T,\rho)$ be defined as
\[
\eta_{j}(\T,\rho)=\int_{\al_{{j}-1}}^{\T} \p F_0^{j}(\phi,\f(\phi,\rho))d\phi.
\] Then for $i=1,2,\ldots,k$ and $j=1,2,\ldots,n$ the recurrence \eqref{z11} can be written as follows

\begin{equation*}
\begin{array}{RLL}
z_1^1(\T,\rho)=&\!\!\!\!e^{\eta_1(\T,\rho)}\int_0^{\T} e^{-\eta_1(\phi,\rho)}F_1^1(\phi,\f(\phi,\rho))d\phi,& \vspace{0.3cm}\\

z_1^j(\T,\rho)=&\!\!\!\!e^{\eta_{j}(\T,\rho)}\bigg(z_1^{j-1}(\al_{j-1},\rho)+\int_{\al_{j-1}}^{\T}e^{-\eta_{j}(\phi,\rho)}F_1^{j}(\phi,\f(\phi,\rho))d\phi\bigg),& \mbox{for $j=2$},\ldots\vspace{0.3cm}\\

z_i^1(\T,\rho)=&\!\!\!\!e^{\eta_{1}(\T,\rho)}i!\int_{0}^{\T}e^{-\eta_{1}(\phi,\rho)}\bigg[F_i^{1}(\phi,\f(\phi,\rho))\\
&+\sum_{l=1}^{i-1}\sum_{m=1}^l\frac{1}{l!}\partial ^m F_{i-l}^{1}(\T,\f(\T,\rho)) B_{l,m}(z_1^1,z_2^1,\ldots,z_{l-m+1}^1)\\
&+\sum_{m=2}^i\frac{1}{i!}\partial ^m F_{0}^{1}(\T,\f(\T,\rho))B_{i,m}(z_1^1,z_2^1,\ldots,z_{i-m+1}^1)\bigg] d\phi, &\mbox{for $i=2$},\ldots\vspace{0.3cm}\\

z_i^j(\T,\rho)=&\!\!\!\!e^{\eta_{j}(\T,\rho)}\bigg(z_i^{j-1}(\al_{j-1},\rho)+i!\int_{\al_{j-1}}^{\T} e^{-\eta_{j}(\phi,\rho)}\bigg[F_i^{j}(\phi,\f(\phi,\rho))\\
&+\sum_{l=1}^{i-1}\sum_{m=1}^l\frac{1}{l!}\partial ^m F_{i-l}^{j}(\T,\f(\T,\rho))B_{l,m}(z_1^j,z_2^j,\ldots,z_{l-m+1}^j)\\
&+\sum_{m=2}^i\frac{1}{i!}\partial ^m F_{0}^{j}(\T,\f(\T,\rho))B_{i,m}(z_1^j,z_2^j,\ldots,z_{i-m+1}^j)\bigg] d\phi\bigg),& \mbox{for $i,j=2$},\ldots.
\end{array}
\end{equation*}
\end{proposition}
\begin{proof}
We shall prove this proposition for $i=1,2,\ldots,k$, and $j=1.$ The other cases will follow in a similar way.

\smallskip

For $i=j=1$, the integral equation \eqref{z11} is equivalent to the following Cauchy problem:
\[
\dfrac{\p z_1^1}{\p\T}(\T,\rho)=F_1^1\left(\T,\f(\T,\rho)\right)+\p F_0^1\left(\T,\f(\T,\rho)\right)u \,\,\text{ with }\,\, z_1^1(0,\rho)=0.
\]
Solving the above linear differential equation we get
\[
z_1^1(\T,\rho)=e^{\eta_1(\T,\rho)}\int_0^{\T} e^{-\eta_1(\phi,\rho)}F_1^1(\phi,\f(\phi,\rho))d\phi.
\]

Now for $i=2,\ldots,k$ and $j=1$ the recurrence \eqref{z11} can be written in terms of the partial Bell polynomials as (for more details, see \cite{Novaes2016})
\begin{equation}
\label{Bellzij}
\begin{array}{RL}
z_i^1(\T,\rho)=& i!\int_{0}^{\T}\bigg(F_i^{1}(\phi,\f(\phi,\rho))\\
&+\sum_{l=1}^{i}\sum_{m=1}^l\frac{1}{l!}\partial ^m F_{i-l}^{1}(\phi,\f(\phi,\rho))B_{l,m}(z_1^1,z_2^1,\ldots,z_{l-m+1}^1)\bigg)d\phi.
\end{array}
\end{equation}
We note that the function $z_i^1$ appears in the right hand side of \eqref{Bellzij} only if  $l=i$ and $m=1$. In this case $B_{i,1}(z_1^1,z_2^1,\ldots,z_{i}^1)=z_i^1$ for every $i\geq 1$. So we can rewriting \eqref{Bellzij} as the following integral equation
\begin{equation*}
\begin{array}{RL}
z_i^1(\T,\rho)
=&i!\int_{0}^{\T}\bigg(F_i^{1}(\phi,\f(\phi,\rho))\\
&+\sum_{l=1}^{i-1}\sum_{m=1}^l\frac{1}{l!}\partial ^m F_{i-l}^{1}(\phi,\f(\phi,\rho))B_{l,m}(z_1^1,z_2^1,\ldots,z_{l-m+1}^1)\\
&+\sum_{m=2}^i\frac{1}{i!}\partial ^m F_{0}^{1}(\phi,\f(\phi,\rho))B_{i,m}(z_1^1,z_2^1,\ldots,z_{i-m+1}^1)\\
&+\frac{1}{i!}\p F_0^1(\phi,\f(\phi,\rho)) B_{i,1}(z_1^1,z_2^1,\ldots,z_{i}^1)\bigg) d\phi,
\end{array}
\end{equation*}
which is equivalent to the following Cauchy problem:
\[
\begin{array}{RL}
\frac{\p z_i^1}{\p\T}(\T,\rho)=&i!\left[F_i^{1}(\T,\f(\T,\rho))+\frac{1}{i!}\p F_0^1(\T,\f(\T,\rho)) z_{i}^1\right.\\
&+\sum_{l=1}^{i-1}\sum_{m=1}^l\frac{1}{l!}\partial ^m F_{i-l}^{1}(\T,\f(\T,\rho))B_{l,m}(z_1^1,z_2^1,\ldots,z_{l-m+1}^1)\\
&\left.+\sum_{m=2}^i\frac{1}{i!}\partial ^m F_{0}^{1}(\T,\f(\T,\rho))B_{i,m}(z_1^1,z_2^1,\ldots,z_{i-m+1}^1)\right],\\
 z_i^1(0,\rho)=&0.
\end{array}
\]
Solving the above linear differential equation we obtain the expressions of $z_i^1(\T,\rho),$ for $i=2,\ldots,k,$ given in the statement of  the proposition.
\end{proof}

\section{Proof of the main result}
\label{proofofthemainresult}
In this section we shall present the proof of Theorem \ref{maintheorem}. This proof is based on a preliminary result (see Lemma \ref{lemma}) which expands the solutions of the discontinuous differential equation \eqref{s1} in powers of $\e$.

\smallskip

From hypothesis $(H1)$ the solution $\varphi(\T,\rho)$ of the unperturbed system reads
\begin{equation*}
\varphi(\T,\rho)=
\begin{cases}
\varphi_1(\T,\rho) & \text{if $0=\al_0 \leq \T \leq \al_1$},\\
\vdots \\
\varphi_j(\T,\rho) & \text{if $\al_{j-1} \leq \T \leq \al_j$},\\
\vdots \\
\varphi_n(\T,\rho) & \text{if $\al_{n-1} \leq \T \leq \al_n=2\pi$},
\end{cases}
\end{equation*}
such that, for each $j=1,2,\ldots,n$, $\varphi_j$ is the solution of the unperturbed system with the initial condition $\f_{j}(\al_{j-1},\rho)=\f_{j-1}(\al_{j-1},\rho)$.

Now for $j=1,2,\ldots,n$ let $\xi_j(\T,\T_0,\rho_0,\e)$ be the solution of the discontinuous differential equation \eqref{star} such that $\xi_j(\T_0,\T_0,\rho_0,\e)=\rho_0.$  We then define the recurrence
\[
 r_j(\T,\rho,\e)=\xi_j(\T,\al_{j-1},r_{j-1}(\al_{j-1},\rho,\e),\e),\quad j=2,\ldots,n,
 \]
 with initial condition  $r_1(\T,\rho,\e)=\xi_1(\T,0,\rho,\e).$ From hypothesis $(H1)$ it is easy to see that each $r_j(\T,\rho,\e)$ is defined for every $\T\in[\al_{j-1},\al_j]$. Therefore $r(\cdot,\rho,\e):[0,2\pi] \to \R$ defined as
\begin{equation*}
\label{sol}
r(\T,\rho,\e)=
\begin{cases}
r_1(\T,\rho,\e) & \text{if $0=\al_0 \leq \T \leq \al_1$},\\
r_2(\T,\rho,\e) & \text{if $\al_1 \leq \T \leq \al_2$},\\
\vdots \\
r_j(\T,\rho,\e) & \text{if $\al_{j-1} \leq \T \leq \al_j$},\\
\vdots \\
r_n(\T,\rho,\e) & \text{if $\al_{n-1} \leq \T \leq \al_n=2\pi$},
\end{cases}
\end{equation*}
is the solution of the differential equation \eqref{s1} such that $r(0,\rho,\e)=\rho$. Moreover the equalities hold
\begin{equation}\label{initial}
r_1(0,\rho,\e)=\rho \,\, \text{and} \,\, r_j(\al_{j-1},\rho,\e)=r_{j-1}(\al_{j-1},\rho,\e),
\end{equation}
for $j=1,2,\ldots, n$. Clearly $r_j(\T,\rho,0)=\varphi_j(\T,\rho)$ for all $j=1,2,\ldots,n$.

\begin{lemma}\label{lemma}
 For $j \in \{1,2,\ldots,n\}$ and $\T_{\rho}^{j}>\al_j,$ let $r_j(\cdot,\rho,\e):[\al_{j-1},\T_{\rho}^{j})$ be the solution of \eqref{star}. Then
\[
r_j(\T,\rho,\e)=\f_j(\T,\rho)+\sum_{i=1}^k \frac{\e^i}{i!} z_i^j(\T,\rho) + \CO(\e^{k+1}),
\]
where $z_i^j(\T,\rho)$ is defined in \eqref{z11}.
\end{lemma}
%

\begin{proof}
Fix $j\in\{1,2,\ldots,n\}$, from the continuity of the solution $r_j(\T,\rho,\e)$ and by the compactness of the
set $[\al_{j-1},\al_{j}]\times\overline{D}\times[-\e_0,\e_0]$ it is easy to obtain that
\[
\int_{\al_{j-1}}^{\T} R^j(\T,r_j(\T,\rho,\e),\e)ds= \CO(\e), \quad \T\in [\al_{j-1},\al_{j}].
\]
Thus integrating the differential equation \eqref{star} from $\al_{j-1}$ to $\T$, we get
\begin{equation}
\label{fundlemma}
r_j(\T,\rho,\e) = r_j(\al_{j-1},\rho,\e) + \sum_{i=0}^k \e^i \int_{\al_{j-1}}^{\T} F_i^j(\phi,r_j(\phi,\rho,\e)) d\phi + \CO(\e^{k+1}).
\end{equation}
Note that in the above expression the value of the initial condition  $r_j(\al_{j-1},\rho,\e)$ is not substituted yet.

\smallskip

In the sequel we shall expand the right hand side of the above equality in Taylor series in $\e$ around $\e=0$. To do that we first recall the Fa\'{a} di Bruno's Formula about the $l$-th derivative of a composite
function. Let $g$ and $h$ be sufficiently smooth
functions then
\[
\dfrac{d^l}{d\al^l}g(h(\al))=
\sum_{S_l} \dfrac{l!}{b_1!\,b_2!2!^{b_2}\cdots
b_l!l!^{b_l}}g^{(L)} (h(\al))\prod_{j=1}^l\left(h^{(j)}(\al)\right)^{b_j},
\]
where $S_l$ is the set of all $l$-tuples of non-negative integers
$(b_1,b_2,\cdots,b_l)$ satisfying $b_1+2b_2+\cdots+lb_l=l$, and
$L=b_1+b_2+\cdots+b_l$. So expanding $F_i^j(\phi,r_j(\phi,\rho,\e))$  in Taylor series in $\e$ around $\e=0$ we obtain
\begin{equation}
\label{s7}
\begin{array}{RL}
F_i^j(\phi,r_j(\phi,\rho,\e))=&F_i^j(\phi,r_j(\phi,\rho,0)) \\
&+ \sum_{l=1}^{k-i} \frac{\e^l}{l!} \left(\frac{\partial^l}{\partial \e^l} F_i^j(\phi,r_j(\phi,\rho,\e))\right) \Big|_{\e=0} + \CO(\e^{k-i+1}).
\end{array}
\end{equation}
From the Fa\'{a} di Bruno's Formula we compute
\begin{equation}
\label{s8}
\begin{array}{RL}
\frac{\partial^l}{\partial \e^l} F_i^j(\phi,r_j(\phi,\rho,\e)) \Big|_{\e=0} = &\sum_{S_l}\dfrac{l!}{b_1!\,b_2!2!^{b_2}\cdots
b_l!l!^{b_l}}\\
&\cdot \partial ^L F_i^j(\phi,\f_j(\phi,\rho))\prod_{m=1}^l w_m^j(\phi,\rho)^{b_m},
\end{array}
\end{equation}
where
\[
w_m^j(\phi,\rho)=\frac{\p^m}{\p \e^m}r_j(\phi,\rho,\e)\Big|_{\e=0}.
\]
Substituting \eqref{s8} in \eqref{s7} we have

\begin{equation}
\label{s10}
\begin{array}{RL}
F_i^j(\phi,r_j(\phi,\rho,\e))  =&  F_i^j(\phi,\f_j(\phi,\rho))\\
& + \sum_{l=1}^{k-i}  \sum_{S_l}\dfrac{\e^l}{b_1!\,b_2!2!^{b_2}\cdots
b_l!l!^{b_l}} \partial ^L F_i^j(\phi,\f_j(\phi,\rho)) \prod_{m=1}^l w_m^j(\phi,\rho)^{b_m},
\end{array}
\end{equation}
for $i=0,1,...,k-1$. Moreover for $i=k$ we have that

\begin{equation}
\label{s11}
F_k^j(\phi,r_j(\phi,\rho,\e))= F_k^j(\phi,\f_j(\phi,\rho)) + \CO(\e).
\end{equation}

Substituting \eqref{s10} and \eqref{s11} in \eqref{fundlemma} we get
\begin{equation}
\label{delta}
\begin{array}{RL}
r_j(\T,\rho,\e)=
&r_{j}(\al_{j-1},\rho,\e) + \int_{\al_{j-1}}^\T\Bigg( \sum_{i=0}^k \e^i F_i^j(\phi,\f_j(\phi,\rho)) d\phi\\
&+\sum_{i=0}^{k-1} \sum_{l=1}^{k-i} \e^{l+i} \sum_{S_l}\dfrac{1}{b_1!\,b_2!2!^{b_2}\cdots
b_l!l!^{b_l}}\\
&\cdot \partial ^L F_i^j(\phi,\f_j(\phi,\rho)) \prod_{m=1}^l w_m^j(\phi,\rho)^{b_m}\Bigg)d\phi+ \CO(\e^{k+1}).
\end{array}
\end{equation}
Denote
\[
Q_j(\phi,\rho,\e)=\sum_{i=0}^{k-1} \sum_{l=1}^{k-i} \e^{l+i} \sum_{S_l}\dfrac{1}{b_1!\,b_2!2!^{b_2}\cdots
b_l!l!^{b_l}} \partial ^L F_i^j(\phi,\f_j(\phi,\rho)) \prod_{m=1}^l w_m^j(\phi,\rho)^{b_m}.
\]
After some transformations of the indexes $i$ and $l$ we obtain
\begin{equation}
\label{qj}
Q_j(\phi,\rho,\e)=\sum_{i=1}^{k} \e^{i} \sum_{l=1}^{i} \sum_{S_l}\dfrac{1}{b_1!\,b_2!2!^{b_2}\cdots
b_l!l!^{b_l}} \partial ^L F_{i-l}^j(\phi,\f_j(\phi,\rho)) \prod_{m=1}^l w_m^j(\phi,\rho)^{b_m}.
\end{equation}
Therefore from \eqref{delta} and \eqref{qj} we have
\begin{equation}
\label{rjw}
r_j(\T,\rho,\e)=r_{j}(\al_{j-1},\rho,\e) + \sum_{i=0}^{k}\e^iI_i^j(\T,\rho) + \CO(\e^{k+1}),
\end{equation}
where for $i=0,\ldots,k$ and $j=1,2,\ldots,n$ we are taking
\begin{equation}\label{I0i}
\begin{array}{RL}
I_0^j(\T,\rho)=&\int_{\al_{j-1}}^\T F_0^j(\phi,\f_j(\phi,\rho))d\phi,\quad j=1,2,\ldots\vspace{0.3cm}\\

I_i^j(\T,\rho)=&\int_{\al_{j-1}}^\T\bigg( F_i^j(\phi,\f_j(\phi,\rho)) +\sum_{l=1}^{i}\sum_{S_l}\dfrac{1}{b_1!\,b_2!2!^{b_2}\cdots b_l!l!^{b_l}} \vspace{0.2cm}\\
&\cdot\partial ^L F_{i-l}^j(\phi,\f_j(\phi,\rho)) \prod_{m=1}^l w_m^j(\phi,\rho)^{b_m}\bigg) d\phi, \quad i,j=1,2,\ldots
\end{array}
\end{equation}
Note that for $i=1,\ldots,k$ and $j=2,\ldots,n$ the following recurrence holds
\begin{equation}
\label{r1}
\begin{array}{RL}
w_i^j(\T,\rho)=&\dfrac{\partial^i }{\partial \e^i}r_j(\T,\rho,\e) \Big|_{\e=0}\vspace{0.2cm}\\
=&\frac{\partial^i }{\partial \e^i}r_{j-1}(\al_{j-1},\rho,\e)\Big|_{\e=0} + i! I_i^j(\T,\rho)\vspace{0.2cm}\\
=&w_i^{j-1}(\al_{j-1},\rho) + i!I_i^j(\T,\rho),
\end{array}
\end{equation}
with the initial condition
\begin{equation}\label{r2initial}
w_i^1(\T,\rho)=\dfrac{\partial^i r_1}{\partial \e^i}(\T,\rho,\e) \Big|_{\e=0}
=\dfrac{\partial^i}{\partial \e^i}\left(\rho + \sum_{q=0}^k \e^q I_q^1(\T,\rho) \right)\Bigg|_{\e=0}
=i! I_i^1(\T,\rho).
\end{equation}
Putting \eqref{r1} and \eqref{r2initial} together we obtain
\begin{equation*}
w_i^j(\T,\rho)=i! \left(I_i^1(\al_1,\rho) + I_i^2(\al_2,\rho) + \cdots + I_i^{j-1}(\al_{j-1},\rho)+I_i^j(\T,\rho)\right).
\end{equation*}
for $i=1,2,\ldots,k$ and $j=1,2,\ldots,n.$

\begin{claim}
\label{claim1}
For $j=1,2,\ldots,n$ we have
\[
r_j(\T,\rho,\e)=\f_j(\T,\rho)+\sum_{i=1}^k \frac{\e^i}{i!} w_i^j(\T,\rho) + \CO(\e^{k+1}).
\]
\end{claim}

This claim will be proved by induction on $j$. Let $j=1$. Since $\varphi_1$ is the solution of \eqref{star} for $\e=0$ and $j=1$ with the initial condition $\f_{1}(0,\rho)=\rho$ we get
\begin{equation*}
\label{f0}
\f_1(\T,\rho)=\rho+\int_0^{\T} F_0^1(\T,\f_1(\phi,\rho)) d\phi.
\end{equation*}
Hence from \eqref{rjw}, \eqref{initial} and \eqref{r2initial} it follows that
\begin{equation*}
\begin{split}
r_1(\T,\rho,\e)&=\rho+\sum_{i=0}^k \e^i I_i^1(\T,\rho) + \CO(\e^{k+1})\\
&=\rho+\int_0^{\T} F_0^1(\T,\f_1(\phi,\rho)) d\phi +\sum_{i=1}^k \frac{\e^i}{i!} w_i^1(\T,\rho) + \CO(\e^{k+1})\\
&=\f_1(\T,\rho) + \sum_{i=1}^k \frac{\e^i}{i!} w_i^1(\T,\rho) + \CO(\e^{k+1}).
\end{split}
\end{equation*}
Therefore the claim is proved for $j=1$.

Now using induction we shall prove the claim for $j=j_0$ assuming that it holds for $j=j_0-1,$ that is
\begin{equation}\label{indhyp}
r_{j_0-1}(\T,\rho,\e)=\f_{j_0-1}(\T,\rho)+\sum_{i=1}^k \frac{\e^i}{i!} w_i^{j_0-1}(\T,\rho) + \CO(\e^{k+1}).
\end{equation}

Since $\varphi_{j_0}$ is the solution of \eqref{star} for $\e=0$ and $j=j_0$ with the initial condition $\f_{j_0}(\al_{j_0-1},\rho)=\f_{j_0-1}(\al_{j_0-1},\rho)$ we get
\begin{equation}
\label{fjj}
\f_{j_0}(\T,\rho)=\f_{j_0-1}(\al_{j_0-1},\rho) + \int_{\al_{j_0-1}}^{\T} F_0^1(\T,\f_j(\phi,\rho))d\phi=\f_{j_0-1}(\al_{j_0-1},\rho)+I_0^{j_0}(\T,\rho).
\end{equation}
From \eqref{rjw}, \eqref{initial} and \eqref{r1} we have
\begin{equation*}
\begin{split}
r_{j_0}(\T,\rho,\e)&=r_{{j_0}-1}(\al_{{j_0}-1},\rho,\e) + \sum_{i=0}^k \e^i I_i^{j_0}(\T,\rho) + \CO(\e^{k+1})\\
&=r_{{j_0}-1}(\al_{{j_0}-1},\rho,\e) + I_0^{j_0}(\T,\rho)+\sum_{i=1}^k \e^i \frac{w_i^{j_0}(\T,\rho)-w_i^{j_0-1}(\al_{j-1},\rho)}{i!}+ \CO(\e^{k+1}).\\
\end{split}
\end{equation*}
Finally using \eqref{indhyp} and \eqref{fjj} the above expression becomes
\begin{equation*}
\begin{array}{RL}
r_{j_0}(\T,\rho,\e)=&\f_{j_0-1}(\al_{j_0-1},\rho)  + I_0^{j_0}(\T,\rho)+ \sum_{i=1}^k \frac{\e^i}{i!} w_i^{{j_0}-1}(\al_{j_0-1},\rho)\\
&+ \sum_{i=1}^k \frac{\e^i}{i!} (w_i^{j_0}(\T,\rho)-w_i^{{j_0}-1}(\al_{j_0-1},\rho))+\CO(\e^{k+1})\\
=&\f_{j_0}(\T,\rho)+ \sum_{i=1}^k \frac{\e^i}{i!} w_i^{j_0}(\T,\rho) + \CO(\e^{k+1}).\\
\end{array}
\end{equation*}
This proves the Claim \ref{claim1}.

\smallskip

The proof of Lemma \ref{lemma} ends by proving the following claim.
\begin{claim}
The equality  $w_i^j=z_i^j$ holds for $i=1,2,\ldots,k$ and $j=1,2,\ldots,n.$
\end{claim}

Computing the derivative in the variable $\T$ of the expressions \eqref{z11} and \eqref{r2initial}, for $i=j=1$, we see, respectively,  that the functions $z_1^1(\T,\rho)$ and $w_1^1(\T,\rho)$  satisfy the same differential equation. Moreover for each $i=1,2,\ldots,k$ the integral equations \eqref{z11} and \eqref{r1} (and the equivalent differential equations), which provides respectively  $z_i^j$ and $w_i^j$, are defined by the same recurrence for $j=2,\ldots,n$. Therefore we conclude that the functions  $z_i^j(\T,\rho)$ and $w_i^j(\T,\rho)$ satisfy the same differential equations for  $i=1,2,\ldots,k$ and $j=1,2,\ldots,n$.

\smallskip

It remains to prove that their initial conditions coincide. Let $i\in\{1,2,\ldots,k\}$. For $j=1$ we have from \eqref{r2initial} and \eqref{z11} that $w_i^1(0,\rho)=0=z_i^1(0,\rho)$. For $j=2,\ldots,n$  the initial conditions are defined by the recurrence $z_i^j(\al_{j-1},\rho)=z_i^{j-1}(\al_{j-1},\rho)$ (see \eqref{z11}), which is the same recurrence for the initial conditions of $w_i^j(\al_{j-1},\rho)$. Indeed from \eqref{r1} and \eqref{I0i} we see that for $j=2,\ldots,n$ we have $w_i^j(\al_{j-1},\rho)=w_i^{j-1}(\al_{j-1},\rho)+i!I_i^j(\al_{j-1},\rho)=w_i^{j-1}(\al_{j-1},\rho)$. Therefore $z_i^{j}(\al_{j-1},\rho)=w_i^{j}(\al_{j-1},\rho)$ for every $i=1,2,\ldots,k$ and $j=1,2,\ldots,n$.

\smallskip

Hence Claim 2 follows from the uniqueness property of the solutions of the differential equations.
\end{proof}

Now we are ready to prove Theorem \ref{maintheorem}.

\begin{proof}[Proof of Theorem \ref{maintheorem}]
Since $\f(\T,\rho)$ is $2\pi$-periodic, using Lemma \ref{lemma} we have
\begin{equation*}
\begin{split}
r_n(2\pi,\rho,\e)&= \f_n(2\pi,\rho)+\sum_{i=1}^k \frac{\e^i}{i!} z_i^n(2\pi,\rho) + \CO(\e^{k+1})\\
&=\rho+\sum_{i=1}^k \frac{\e^i}{i!} z_i^n(2\pi,\rho) + \CO(\e^{k+1}).
\end{split}
\end{equation*}
Therefore from \eqref{fpromediada} the following equality holds
\begin{equation}
\label{final}
r_n(2\pi,\rho,\e) = \rho + \e f_1(\rho)+\e^2 f_2(\rho)+\cdots+\e^k f_k(\rho) + \CO(\e^{k+1}).
\end{equation}

Consider the displacement function
\begin{equation*}
f(\rho,\e)=r(2\pi,\rho,\e)-\rho=r_n(2\pi,\rho,\e)-\rho.
\end{equation*}
Clearly for some $\e=\bar{\e} \in (-\e_0,\e_0)$ discontinuous differential equation \eqref{s1} admits a periodic solution passing through $\bar{\rho} \in D$ if and only if $f(\bar{\rho},\bar{\e})=0$. From \eqref{final} we have that
\[
f(\rho,\e)=\sum_{i=1}^k \e^i f_i(\rho) + \CO(\e^{k+1}).
\]

By hypotheses $f_l(\rho^*)=0$ and $f_l'(\rho^*)\neq0$. Using the Implicit Function Theorem for the function $\mathcal{F}(\rho,\e)=f(\rho,\e)/\e^l$ we guarantee the existence of a differentiable function $\rho(\e)$ such that $\rho(0)=\rho^*$ and $f(\rho(\e),\e)=0$ for every $|\e|\neq 0$ sufficiently small. This completes the proof of Theorem \ref{maintheorem}.
\end{proof}

\section{Examples}\label{examples}

In this section we present three applications of our main result (Theorem \ref{maintheorem}). In the first two examples (subsections 4.1 and 4.2) we use the averaged functions \eqref{fpromediada} up to order $7$ to provide lower bounds for the maximum number of limit cycles admitted by some piecewise linear systems with four zones. The first system is a piecewise linear perturbation of the linear center $(\dot x,\dot y)=(-y,x)$, and the second one is a piecewise linear perturbation of a  discontinuous piecewise constant center.  As usual, the expressions of the higher order averaged functions are extensive (see \cite{LliItiNov2015,LliNovTeiN2014}), so we shall omit them here. We emphasize that our goal in these first two examples, by taking particular classes of perturbations, is to illustrate the using of the higher order averaged functions.

\smallskip

In the third example we study the quadratic isochronous center $(\dot x,\dot y)=(-y+x^2,x+ xy)$ perturbed inside a particular family of piecewise quadratic system with $n$ zones. Using the first order averaged function \eqref{fpromediada} we provide lower bounds, depending on $n$, for the maximum number of limit cycles admitted by this system. We emphasize that our goal in this last example, again by taking a particular class of perturbation, is to illustrate the using of Theorem \ref{maintheorem} to study discontinuous piecewise nonlinear system with many zones.

\smallskip

The next proposition, proved in \cite{Gasull}, is needed to deal with our examples.

\begin{proposition}
\label{functionsli}Consider $n$ linearly independent functions $h_i:I \to \R$, $i=1,2,\ldots,n$.
\begin{itemize}
\item [(i)] Given $n-1$ arbitraries values of $a_i \in I$, $i=1,2,\ldots,n-1$ there exist $n$ constants $\beta_k$, $i=1,2,\ldots,n$ such that
\begin{equation}
\label{gazulproposition}
h(x) \doteq \sum_{k=1}^n \beta_kh_k(x),
\end{equation} is not the zero function and $h(a_i)=0$ for $i=1,2,\ldots,n-1$.
\item[(ii)] Furthermore, if all $h_i$ are analytical functions on $I$ and there exists $j \in \{1,2,\ldots,n\}$ such that $h_j |_I$ has constant sign, it is possible to get an $h$ given by \eqref{gazulproposition}, such that it has at least $n-1$ simple zeroes in $I$.
\end{itemize}
\end{proposition}

\smallskip
\subsection{Nonsmooth perturbation of the linear center}\label{example1}
The bifurcation of limit cycles from smooth and nonsmooth perturbations of the linear center $(\dot x,\dot y)=(-y,x)$ is a fairly studied problem in the literature, see for instance \cite{BPT,CT,GLN,LT,N}. Here we apply our main result (Theorem \ref{maintheorem}) to study these limit cycles when the linear center  is perturbed inside a particular of piecewise linear system with $4$ zones. Following the notation introduced in subsection 1.2 we take
\begin{equation}
\label{linearcenter1}
\begin{array}{l}
X_0^j(x,y)=\big(-y, x\big), \,\, \text{for}\,\, j=1,\ldots,n,\,\, \text{and}\vspace{0.1cm}\\
X_i^j(x,y)=\big(a_{ij} x + b_{j},0\big),\,\,\text{for}\,\, j=1,\ldots,n,\,\,\text{and}\,\, i=1,\ldots,k.
\end{array}
\end{equation}
with $a_{ij}, b_{ij} \in \R$ for all $i,j$. We consider the discontinuous piecewise differential system $(\dot x,\dot y)^T=Z_{\X,\al}(x,y;\e)$ (see \eqref{planar}) where $\X=\big(X_1,\ldots,X_4)$ (see \eqref{Xj}) and $\al=(\al_0,\al_1,\al_2,\al_3)=(0, \pi/2,  \pi, 3\pi/2).$

\smallskip

First of all, in order to apply our main result (Theorem \ref{maintheorem}) to study the limit cycles of $(\dot x,\dot y)^T=Z_{\X,\al}(x,y;\e)$, we shall write it into the standard form \eqref{s1}. To do that we consider the polar coordinates $x=r \cos \T$, $y=r \sin \T$. So the set of discontinuity becomes $\Sigma=\{\T=0\}\cup\{\T=\al_1\}\cup \{\T=\al_{2}\}\cup \{\T=\al_{3}\}$ and  in each sector $C_j$ (see \eqref{planarsector}), $j=1,2,3,4,$  the differential system $(\dot x,\dot y)^T=Z_{\X,\al}(x,y;\e)$ reads
\begin{equation*}
\label{linearcenterpolarj}
\begin{array}{RL}
\dot r(t)=&\sum_{i=1}^7\e^i(a_{ij}r\cos^2\T + b_{ij}\cos \T),\\
\dot \T(t)=&1-\frac{1}{r}\sum_{i=1}^7\e^i(a_{ij}r\cos \T \sin \T + b_{ij} \sin \T).
\end{array}
\end{equation*}
Note that $\dot \T(t)\neq0$ for $|\e|$ sufficiently small, thus we can take $\T$ as the new independent time variable by doing $r'(\T)=\dot r(t)/\dot \T(t)$. Then
\begin{equation}
\label{ex1polar}
r'(\T)=\dfrac{\dot r(t)}{\dot \T(t)}=\sum_{i=1}^7\e^i F_i^j(\T,r) + \e^{k+1}R^j(\T,r,\e),\quad \text{for}\quad j=1,2,3,4,
\end{equation}
where $F_i^j$ is the coefficient of $\e^i$ in the Taylor series in $\e$ of $\dot r(t)/\dot \T(t)$ around $\e=0$.

\smallskip

From here we shall use the averaged functions $\eqref{fpromediada}$ up to order $7$ to study the isolated periodic solutions of the piecewise differential equation defined by \eqref{ex1polar} or, equivalently, the limit cycles of the piecewise differential system $(\dot x,\dot y)^T=Z_{\X,\al}(x,y;\e)$ defined by \eqref{linearcenter1}. As we have said before, due to the complexity of the expressions of the higher order averaged functions we shall not provided them explicitly.  So we first describe the methodology to obtain lower bounds for the number of their zeros, and consequently for the number of limit cycles of \eqref{linearcenter1}.

\smallskip

Assume that one have computed the list of averaged functions $f_i$, $i=1,\ldots,k,$ and that they are polynomials. The first step is to established a lower bound for the number of zeros that $f_1$ can have. To do that, one can build a vector $M_1$ where each entry $s$ of $M_1$ is given by the coefficient of $r^s$ of the function $f_1$. Clearly $M_1$ is a function on the parameter variable $v_1=\{a_{1j}:\, j=1,\ldots,4\}\cup\{b_{1j}:\, j=1,\ldots,4\}$. So taking the derivative $D_{v_1} M_1$, a lower bound for the number of zeros of $f_i$ will be given by the rank of the matrix $D_{v_1} M_1$ decreased by $1$. For instance, in our first example system \eqref{ex1polar}, the averaged function $f_1$ reads
\begin{equation*}
\label{f1linearcenter}
\begin{array}{RL}
f_1(r)=&\int_{0}^{\frac{\pi}{2}} F_1^1(\T,r) d\T + \int_{\frac{\pi}{2}}^{\pi} F_1^2(\T,r) d\T + \int_{\pi}^{\frac{3\pi}{2}} F_1^3(\T,r) d\T + \int_{\frac{3\pi}{2}}^{2\pi} F_1^4(\T,r) d\T\\
=&\dfrac{\pi}{4}   r (a_{11}+a_{12}+a_{13}+a_{14})+b_{11}-b_{12}-b_{13}+b_{14}.
\end{array}
\end{equation*}
Clearly $f_1$ has at most one positive root and there exists parameters $a_{1j}$'s and $b_{1j}$'s  for which this zero exists. In this case
\[
\begin{array}{c}
M_1=\Big(b_{11}-b_{12}-b_{13}+b_{14}\,,\,\dfrac{1}{4} \pi  (a_{11}+a_{12}+a_{13}+a_{14})\Big),\vspace{0.2cm}\\

D_{v_1}M_1=\left(
     \begin{array}{cccccccc}
       0 & 0 & 0 & 0 & 1 & -1 & -1 & 1\\
       \dfrac{\pi}{4}  & \dfrac{\pi}{4}  & \dfrac{\pi}{4}  & \dfrac{\pi}{4}  & 0 & 0 & 0 & 0
     \end{array}
   \right).
   \end{array}
   \]
We note that the matrix $D_{v_1}M_1$ has maximum rank $2$. Applying Theorem \ref{maintheorem} for $l=1$ we obtain at least one limit cycle for the differential system \eqref{linearcenter1}.

The next step is to chose parameters to assure that $f_1(r) \equiv 0$. In our example $a_{11}=-(a_{12}+a_{13}+a_{14})$ and $b_{11}=b_{12}+b_{13}-b_{14}$. To continue the analysis we repeat the above procedure: build a vector $M_2$ where each entry $s$ of $M_2$ is given by the coefficient of $r^s$ of the function $f_2$; define the parameter vector  $v_2=\{a_{1j}:\, i=1,2,\, j=1,\ldots,4\}\cup\{b_{1j}:\, i=1,2,\,j=1,\ldots,4\}$; and take the derivative $D_{v_2} M_2$. Again a lower bound for the number of zeros of $f_2$ is given by the rank of the matrix $D_{v_2} M_2$ decreased by $1$. In our example
\begin{equation*}
\label{f2linearcenter}
\begin{array}{RL}
f_2(r)=&r^2\left[\pi  (a_{21}+a_{22}+a_{23}+a_{24})+2 (a_{12}+a_{13}) (a_{13}+a_{14})\right]\\
&+r\left[\pi(a_{12}+a_{13})(b_{13}-b_{14})-4(a_{14}b_{12}+(a_{12}+a_{14})b_{13}\right.\\
&\left.+a_{13}(b_{12}+2b_{13}-b_{14})-a_{12}b_{14}-b_{21}+b_{22}+b_{23}-b_{24})\right]\\
&+4(b_{12}+b_{13})(b_{13}-b_{14}).
\end{array}
\end{equation*}
The function $f_2$ is a polynomial of degree $2$ in $r$. It is easy to see that the matrix $D_{v_2}M_2$ has maximum rank again, that is $3$. Applying Theorem \ref{maintheorem} for $l=2$ we obtain at least two limit cycles for the differential system \eqref{linearcenter1}.

\smallskip

In general, after estimating a lower bound for the number of zeros of $f_{l-1}$ we chose parameters to assure that $f_{l-1}(r) \equiv 0$. Then we follow the above steps: build a vector $M_l$ where each entry $s$ of $M_l$ is given by the coefficient of $r^s$ of the function $f_l$; define the parameter vector $v_l=\{a_{ij}:\, i=1,\ldots l\,, j=1,\ldots,4\}\cup\{b_{ij}:\, i=1,\ldots l\,, j=1,\ldots,4\}$; and take the derivative $D_{v_l} M_l$. As above a lower bound for the number of zeros of $f_l$ will be given by the rank of the matrix $D_{v_l} M_l$ decreased by $1$.

\smallskip

In what follows, using the procedure described above, we provide a table showing the lower bound $N(l)$, $l=1,\ldots,7,$ for the maximum number of limit cycles of the piecewise differential system $(\dot x,\dot y)^T=Z_{\X,\al}(x,y;\e)$, defined by \eqref{linearcenter1}, obtained by studying the averaged function of order $l$.
\begin{center}
\begin{tabular}{|c|c|c|c|c|c|c|c|}
  \hline
  $l$ & $1$ & $2$ & $3$ & $4$ & $5$ & $6$ & $7$ \\
  \hline
  $N(l)$ & $1$ & $2$ & $2$ & $3$ & $3$ & $3$ & $3$ \\
  \hline
\end{tabular}
\end{center}

\subsection{Nonsmooth perturbation of a piecewise constant center}\label{example2}
Consider the discontinuous piecewise constant differential system
\begin{equation}
\label{constantcenter}
(\dot x,\dot y)^T=X(x,y)=
\begin{cases}
X_1(x,y)& \text{if $x>0$ and $y>0$},\\
X_2(x,y)& \text{if $x<0$ and $y>0$},\\
X_3(x,y)& \text{if $x<0$ and $y<0$},\\
X_4(x,y)& \text{if $x>0$ and $y<0$},
\end{cases}
\end{equation}
where
\begin{equation*}
X_1(x,y)=
\begin{cases}
-1 + \sum_{i=1}^7 \e^i(a_{i1} x + b_{i1}),\\
\,\,\,\,1,\\
\end{cases}\quad
X_2(x,y)=
\begin{cases}
-1 + \sum_{i=1}^7 \e^i(a_{i2} x + b_{i2}),\\
-1,\\
\end{cases}
\end{equation*}

\begin{equation*}
X_3(x,y)=
\begin{cases}
\,\,\,\,1 + \sum_{i=1}^7 \e^i(a_{i3} x + b_{i3}),\\
-1,\\
\end{cases}\quad
X_4(x,y)=
\begin{cases}
1 + \sum_{i=1}^7 \e^i(a_{i4} x + b_{i4}),\\
1,\\
\end{cases}
\end{equation*}
with $a_{ij}, b_{ij} \in \R$ for all $i \in \{1,2,\ldots,7\}$ and $j \in \{1,2,3,4\}$.

\smallskip

First of all, in order to apply our main result (Theorem \ref{maintheorem}) to study the limit cycles of the differential system \eqref{constantcenter}, we shall write it into the standard form \eqref{s1}. Again, to do that we consider polar coordinates $x=r \cos \T$, $y=r \sin \T$. So the set of discontinuity becomes $\Sigma=\{\T=0\}\cup\{\T=\al_1\}\cup \{\T=\al_{2}\}\cup \{\T=\al_{3}\}, $ with $\al_0=0, \al_1=\pi/2,  \al_2=\pi, \al_3=3\pi/2,$ and $\al_4=2\pi$, and  for each $j=1,2,3,4$ the differential system $(\dot x,\dot y)=X_j(x,y)$ reads
\begin{equation*}
\label{linearcenterpolarj2}
\begin{split}
&\dot r(t)=g_j(\T) + \sum_{i=1}^7\e^i(a_{ij}r\cos^2\T + b_{ij}\cos \T),\\
&\dot \T(t)=\frac{1}{r}\left(\widehat{g}_j(\T) - \sum_{i=1}^7\e^i(a_{ij}r\cos \T \sin \T + b_{ij} \sin \T)\right),
\end{split}
\end{equation*}
where
\begin{equation*}
\begin{array}{RRL}
&g_1(\T)=\sin \T - \cos \T \quad \quad \quad &\widehat{g}_1(\T)=\sin \T + \cos \T,\\
&g_2(\T)=-(\sin \T + \cos \T) \quad \quad \quad &\widehat{g}_2(\T)=\sin \T - \cos \T,\\
&g_3(\T)=-\sin \T + \cos \T \quad \quad \quad &\widehat{g}_3(\T)=-(\sin \T + \cos \T),\\
&g_4(\T)=\sin \T + \cos \T \quad \quad \quad &\widehat{g}_4(\T)=-\sin \T + \cos \T.
\end{array}
\end{equation*}
Note that for each $j=1,2,3,4$ and $\al_{j-1}\leq\T\leq\al_j$, we have that $\dot \T(t)\neq0$ for $|\e|$ sufficiently small, thus we can take $\T$ as the new independent time variable by doing $r'(\T)=\dot r(t)/\dot \T(t)$. Then
\begin{equation}
\label{ex3polar}
r'(\T)=\dfrac{\dot r(t)}{\dot \T(t)}=\sum_{i=0}^7\e^i F_i^j(\T,r) + \e^{k+1}R^j(\T,r,\e),
\end{equation}
where $F_i^j$ is the coefficient related to $\e^i$ in Taylor series in $\e$ of $\dot r(t)/\dot \T(t)$ around $\e=0$.

\smallskip

From here we shall use the averaged functions $\eqref{fpromediada}$ up to order $7$ to study the isolated periodic solutions of the piecewise differential equation defined by \eqref{ex3polar} or, equivalently, the limit cycles of the piecewise differential system  \eqref{constantcenter}. Following the same methodology described in subsection \ref{example1}, we provide a table showing the lower bound $N(l)$, $l=1,\ldots,7,$ for the maximum number of limit cycles of \eqref{constantcenter} obtained by studying the averaged function of order $l$.

\begin{center}
\begin{tabular}{|c|c|c|c|c|c|c|c|}
  \hline
  $l$ & $1$ & $2$ & $3$ & $4$ & $5$ & $6$ & $7$ \\
  \hline
  $N(k)$ & $1$ & $2$ & $2$ & $2$ & $2$ & $2$ & $2$ \\
  \hline
\end{tabular}
\end{center}

\subsection{Nonsmooth perturbation of an isochronous quadratic center}
In this section we consider the quadratic isochronous center $(\dot x,\dot y)=(-y+x^2,x+ xy)$ perturbed inside a class of piecewise quadratic system with $n$ zones. Following the notation introduced in subsection 1.2 we take
\begin{equation*}\label{quadraticcenter1}
\begin{array}{l}
X_0^1(x,y)=\big(-y+x^2, x+xy\big), \,\, \text{for}\,\, j=1,\ldots,n,\,\, \text{and}\vspace{0.1cm}\\
 X_i^1(x,y)=\big(a_{j} x^2 + b_{j}x + c_{j},0\big),\,\,\text{for}\,\, j=1,\ldots,n,
\end{array}
\end{equation*}
where $a_{j}, b_{j}$ and $c_{j}$ are real numbers for all $j \in \{1,2,\ldots,n\}$. We consider the discontinuous piecewise differential system $(\dot x,\dot y)^T=Z_{\X,\al}(x,y;\e)$ (see \eqref{planar}) where $\X=\big(X_1,\ldots,X_n)$ (see \eqref{Xj}) and $\al=(\al_j)_{j=0}^{n-1}=(2j\pi/n)_{j=0}^{n-1}.$

\smallskip

As before, in order to apply our main result (Theorem \ref{maintheorem}) to study the limit cycles of $(\dot x,\dot y)^T=Z_{\X,\al}(x,y;\e)$, we shall write it into the standard form \eqref{s1}. To do that we consider a first change of coordinates $x=-u/(v-1)$, $y=-v/(v-1)$ (see \cite{CS}). Note that this change keeps fixed all straight lines passing through the origin and therefore does not change the set of discontinuity. In each sector $C_j$ (see \eqref{planarsector}), $j=1,2,3,4,$  the differential system $(\dot x,\dot y)^T=Z_{\X,\al}(x,y;\e)$ reads
\begin{equation}
\label{quadraticcenter2}
\begin{array}{RL}
\dot u =& -v + \e \Bigg(u \bigg(b_{j}-\frac{a_{j}}{v-1}u \bigg)+c_{j}(1-v)\Bigg),\\
\dot v= &u.
\end{array}
\end{equation}
Now, as a second change of variables, we consider the polar coordinates $u=r\cos\T$ and $v=r\sin\T.$ Taking $\T$ as the new independent time variable by doing $r'(\T)=\dot r(t)/\dot \T(t)$, system \eqref{quadraticcenter2} becomes
\[
r'(\T)=\e F^j(\T,r) + \CO(\e^2),
\]
where
\begin{equation*}
\begin{array}{RL}
F^j(\T,r)=&\cos \T \Bigg(c_{j} +r \bigg(-c_{j} \sin \T+\cos \T \bigg(b_{j}+\frac{a_{j} r \cos \T}{1-r \sin \T}\bigg)\bigg)\Bigg).
\end{array}
\end{equation*}
for $j=1,\ldots,n$.

\smallskip

In this new coordinates the piecewise differential system $(\dot x,\dot y)^T=Z_{\X,\al}(x,y;\e)$ reads
\begin{equation}
\label{centroquadraticonzonas}
r'(\T)=\e F(\T,r) + \CO(\e^2),
\end{equation}
 where
\begin{equation*}
F(\T,r)=\sum_{j=1}^{n} \chi_{[\frac{2(j-1)\pi}{n},\frac{2j\pi}{n}]}(\T) F^j(\T,r).
\end{equation*}
Computing the first order averaged function $f_1$ of \eqref{centroquadraticonzonas} we obtain

\begin{equation*}
\label{averagefunctionnzones}
\begin{array}{RL}
f_1(r)=&\sum_{j=1}^{n} \int_{\frac{2(j-1)\pi}{n}}^{\frac{2j\pi}{n}} F^j(\T,r) d\T\\
=&\frac{1}{4}\Bigg[\left(\sum_{j=1}^{n} 4(a_j+cj)\left(\sin \left(\frac{2j\pi}{n}\right)-\sin\left(\frac{2(j-1)\pi}{n}\right)\right)\right)\\
&+r \Bigg(\sum_{j=1}^{n} \left(\frac{4\pi }{n}+\sin\left( \frac{4j\pi }{n}\right)-\sin \left(\frac{4(j-1)\pi}{n}\right)\right)b_j\\
&+(a_j-c_j)\left(\cos\left(\frac{4(j-1)\pi}{n}\right)- \cos\left(\frac{4j\pi}{n}\right)\right)\Bigg)\\
&+\dfrac{(r^2-1)}{r} \left(\sum_{j=1}^{n} 4a_j \ln \left(1-r \sin \left(\frac{2(j-1)\pi}{n}\right)\right)\right)\\
&+\dfrac{(r^2-1)}{r} \left(\sum_{j=1}^{n} -4a_j \ln \left(1-r \sin\left( \frac{2j\pi}{n}\right)\right)\right)\Bigg].
\end{array}
\end{equation*}
Since  $\sin \left(\frac{2(j-1)\pi}{n}\right)=0$ for $j=1,$ and $\sin \left(\frac{2j\pi}{n}\right)=0$ for $j=n$, the above expression simplifies as

\begin{equation*}
\label{averagefunctionnzones2}
\begin{array}{RL}
f_1(r)=&\frac{1}{4}\Bigg[\left(\sum_{j=1}^{n} 4(a_j+cj)\left(\sin \left(\frac{2j\pi}{n}\right)-\sin\left(\frac{2(j-1)\pi}{n}\right)\right)\right)\\
&+r \Bigg(\sum_{j=1}^{n} \left(\frac{4\pi }{n}+\sin \left(\frac{4j\pi }{n}\right)-\sin \left(\frac{4(j-1)\pi}{n}\right)\right)b_j\\
&+(a_j-c_j)\left(\cos\left(\frac{4(j-1)\pi}{n}\right)- \cos\left(\frac{4j\pi}{n}\right)\right)\Bigg)\\
&+\dfrac{(r^2-1)}{r} \left(\sum_{j=2}^{n} 4(a_j-a_{j-1}) \ln \left(1-r \sin \left(\frac{2(j-1)\pi}{n}\right)\right)\right)\Bigg].
\end{array}
\end{equation*}
Note that $f_1$ is written as a linear combination of $n+1$ functions of the family
\[
\CF=\left\{1,r,h_j(r)\doteq\dfrac{(r^2-1)}{r}\ln \left(1-r \sin \left(\frac{2(j-1)\pi}{n}\right)\right):\, j=2,3,\ldots,n\right\}.
\]
It is easy to see that this combination is linearly independent.

\smallskip

Regarding the functions $h_j$'s we have the following properties
\begin{itemize}
\item[{\bf (1)}]  Let $j\in\{2,3,\ldots,n\}$. Then $h_j(r) \equiv 0$  if and only if $n$ is even and $j=1+n/2$.

\smallskip

\item[{\bf (2)}] Let $j_1,j_2\in\{2,3,\ldots,n\}.$ Then $h_{j_1}(r)\equiv h_{j_2}(r)$  if and only if $n$ is even and
$(j_1+j_2-2)\in\{n/2,3n/2\}$.
\end{itemize}

From the above properties we first conclude that if $n$ is odd then the function $f_1$ is a linearly independent combination of $n+1$ linearly independent functions. From Proposition \ref{functionsli} we can find parameters such that $f_1$ has $n$ simple zeros.

\smallskip

If $n=2$ then $f_1(r)=\pi  (b_1+b_2) r/2$ which has no simple positive zeros. From now on we assume that $n$ is even and greater than $2$.  From property {\bf (1)} we already know that $h_{j_0}\equiv 0$ for $j_0=1+n/2$. From property {\bf (2)} it remains  to analyze how many pairs of integers $(j_1,j_2)$, $2\leq j_1< j_2\leq n,$ satisfy the equations $2(j_1+j_2-2)=n$ and $2(j_1+j_2-2)=3n$.

\smallskip

Let $\ov n$ be a positive integer. If $n=4\ov n$ then both equations $2(j_1+j_2-2)=n$ and $2(j_1+j_2-2)=3n$ have $n/4-1$ solutions. If $n=4\ov n+2$  then both equations $2(j_1+j_2-2)=n$ and $2(j_1+j_2-2)=3n$ have $(n-2)/4$ solutions. Therefore we conclude that:
\begin{itemize}
\item If $n=4\ov n$ then $\#\CF=\dfrac{n}{2}+2$;

\smallskip

\item If $n=4\ov n+2$ then $\#\CF=\dfrac{n}{2}+1$;
\end{itemize}

Denote by $N$ the maximum number of limit cycles of $(\dot x,\dot y)^T=Z_{\X,\al}(x,y;\e)$. Applying Proposition \ref{functionsli} and Theorem \ref{maintheorem} we conclude that:
\begin{itemize}
\item [(i)] If $n$ is odd then $N\geq n$;

\smallskip

\item [(ii)] If $n=2$ then $N\geq 0$ (no information!);

\smallskip

\item[(iii)] If $n=4k$ then $N\geq \dfrac{n}{2}+1$;

\smallskip

\item[(iv)] If $n=4k+2$ then $N\geq \dfrac{n}{2}$.
\end{itemize}

\section*{Acknowledgements}

%

The first author is partially supported by a MINECO grant MTM2013-40998-P and an
AGAUR grant number 2014SGR-568.
The second author is supported by FAPESP grants 2015/02517-6 and 2015/24841-0.
The first and the second authors are supported by the European Community FP7-PEOPLE-2012-IRSES-316338 and FP7-PEOPLE-2012-IRSES-318999 grants.
The third author has been supported by a Ph.D. CAPES grant and by CAPES CSF-PVE-88887.
The three authors are also supported by the joint project CAPES-MECD grant PHB-2009-0025-PC.

\end{document}